\documentclass[12pt]{article}

\def\edo{\end{document}	 }

\usepackage{amsmath, amscd, amsfonts, amssymb, amsthm}

\newtheorem{theorem}{Theorem}[section]
\newtheorem{proposition}[theorem]{Proposition}
\newtheorem{claim}{Claim}

\newtheorem{remark}[theorem]{Remark}

\def\rrd{{\mathbb{R}^d}}
\def\call{{\mathcal{L}}}

\def\calf{{\mathcal{F}}}
\def\calm{{\mathcal{M}}}

\def\cald{{\mathcal{D}}}
\def\calx{{\mathcal{X}}}

\def\call{{\mathcal{L}}}

\def\calp{{\mathcal{P}}}

\def\vsp{\vspace*{1,5mm}\\ }

\def\mk{\medskip }
\def\sk{\smallskip }
\def\n{\noindent }
\def\dd{\displaystyle}

\def\barr{\begin{array}}
\def\earr{\end{array}}

\def\bit{\begin{itemize}}
\def\eit{\end{itemize}}

\def\FP{Fokker--Planck}

\def\1{^{-1}}

\def\rr{{\mathbb{R}}}

\def\9{{\infty}}
\def\lbb{{\lambda}}

\def\ov{\overline}
\def\vf{{\varphi}}

\def\ooo{{\Omega}}
\def\pp{{\partial}}
\def\vp{{\varepsilon}}

\def\ff{\forall }
\def\({\left(}
\def\){\right)}
\def\<{\left<}
\def\>{\right>}

\def\NS{Navier--Stokes}
\def\divv{{\rm div}}
\def\curl{{\rm curl}}
\def\D{{\Delta}}

\title{Wasserstein regularity of vorticity solutions  to~the~2D Navier--Stokes equations} 
\author{Viorel Barbu\thanks{Octav Mayer Institute of Mathematics of  Romanian Academy, Ia\c si, Romania.  Email: vbarbu41@gmail.com}}
\date{}

\begin{document}
\maketitle\vspace*{-4mm}
\begin{abstract}
\n We prove herein the absolute continuity in the Wasserstein metric  $W_p$, $1<p<2$, of mild solutions $u$ to $2D$ Navier--Stokes vorticity equations  $u_t-\nu\Delta u+\divv(K(u)u)=0$ on $(0,\9)\times\rr^2$ with $L^1$-initial data in the Wasserstein space $\mathbb{W}_p(\rr^2)$. Herein, $K$ is the 2D Biot--Savart operator. In this way, the vorticity flow $t\to u(t)$ can be identified with an absolutely continuous curve in the Wasserstein space $\mathbb{W}_p(\rr^2)$ relating so its motion with the Monge--Kantorovich optimal transport problem. One obtains also a time variation formula for the Boltzmann--Gibbs entropy of the vorticity flow. These results extend to  finite Radon measure initial data and are of interest to describing the dyna\-mic of 2D vortices and its connection with optimal transport theory.\sk
\end{abstract}
\section{Introduction}\label{s1}
Consider here the 2-D incompressible \NS\ equation
\begin{equation}\label{e1.1}
\barr{ll}
y_t-\nu\D y+(y\cdot\nabla)y=\nabla p&\mbox{in }(0,\9)\times\rr^2,\vsp
\nabla\cdot y=0&\mbox{in }(0,\9)\times\rr^2,\vsp
y(0,x)=y_0(x),\ x\in\rr^2.
\earr\end{equation} \vfill\newpage
\n Let $u=u(t,x)$ denote the vorticity of the velocity field $y=\{y_1,y_2\}$,  
$$u(t,x)=\curl\,y(t,x)=D_1y_2(t,x)-D_2y_1(t,x),\ (t,x)\in(0,\9)\times\rr^2,$$where $D_j=\frac\pp{\pp x_j},\ j=1,2,$ and the symbols  $\nabla,\divv$ refer to spatial derivatives.

Equation \eqref{e1.1} can then be   rewritten as the {\it vorticity equation}
\begin{equation}\label{e1.2}
	\barr{ll}
	u_t-\nu\D u+\divv(yu)=0&\mbox{in }(0,\9)\times\rr^2,\\[1mm]
	u(0,x)=u_0(x)=\curl\,y_0(x),&x\in\rr^2,\earr\end{equation}
while the velocity field $y(t,x)$ can be recovered from the vorticity $u$ via the Biot--Savart formula
\begin{equation}\label{e1.3}
y(t,x)=(\nabla^\bot E*u(t))(x),\ \ff(t,x)\in(0,\9)\times\rr^2,\end{equation}
where $E$ is the fundamental solution of the Laplace operator, i.e.,  $$E(x)=\frac1{2\pi}\,\log|x|,\ x\in\rr^2,$$
and $*$ is the convolution product on $\rr^2$. 
The operator
$$\nabla^\bot E(x)=\frac{(-x_2,x_1)}{2\pi|x|^2},\ x=(x_1,x_2)\in\rr^2\setminus\{0\}$$  
is the Biot--Savart kernel. We set
\begin{equation}\label{e1.5}
K(z)=\nabla^\bot E*z,\ z\in L^m(\rr^2),\ m\in(1,2)\end{equation}
and note (see, e.g., \cite{6}, Lemma 2.2) that
\begin{equation}\label{e1.6}
|K(z)|_{L^q(\rr^2)}\le C|z|_{L^m(\rr^2)},\ \ff z\in L^m(\rr^d), \end{equation}
where $m\in(1,2)$ and $\frac1q=\frac1m-\frac12.$  

Then, we may write \eqref{e1.2} as
\begin{equation}\label{e1.7}
\barr{ll}
u_t-\nu\D u+\divv(uK(u))=0&\mbox{in }(0,\9)\times\rr^2,\vsp
u(0,x)=u_0(x),&x\in\rr^2.\earr\end{equation} 
This is a special case of a so called {\it generalized mean-field \FP\ equation} with locally integrable singular kernel $K$ (see \cite{1aa}, \cite{2}).  

There is an extensive literature on the well-posedness of the vorticity equation \eqref{e1.7} with initial data $u_0$ a Radon measure and, implicitly, on the Navier--Stokes equation \eqref{e1.1} in the spaces $L^p((0,\9)\times\rr^2)$ (see, e.g., \cite{4a}, \cite{5}, \cite{6}, \cite{8}). It should be mentioned also that (see \cite{3}) there is a probabilistic interpretation of equation \eqref{e1.7} in terms of the McKean--Vlasov stochastic dif\-fe\-ren\-tial equation 
\begin{equation}\label{e1.8}
\barr{l}
dX(t)=K(u(t,\cdot))(X(t))dt+\sqrt{2\nu}\ dW(t),\ t\ge0,\vsp 
X(0)=X_0,\earr\end{equation}
defined on a probability space $(\ooo,\calf,\mathbb{P})$ with the normal filtration $(\calf_t)_{t\ge0}$ and 2-D $(\calf_t)$-Brownian motion $W(t)$, where
\begin{equation}\label{e1.9}
u(t,x)dx=\mathbb{P}\circ X(t)\1(dx);\ t>0,\  u_0(dx)=\mathbb{P}\circ X^{-1}_0(dx).\end{equation}
More precisely, if $u$ is a distributional solution to \eqref{e1.7}, then there is a unique strong solution to the stochastic differential equation \eqref{e1.8} such that \eqref{e1.9} holds (see \cite{3}, Theorem 4.5).

If the initial data $u_0$ is the space $\calm(\rr^2)$ of finite Radon measures on $\rr^2$ (in particular, if $u_0\in L^1(\rr^2)$), then there is a mild solution $u$ to the vorticity equation \eqref{e1.7} which is $w^*-\calm(\rr^2)$ continuous on $[0,\9)$, smooth for $t>0$ and for all $r\ge1$ the function $t\to\|u(t)\|_{L^r(\rr^2)}$ is singular of order $t^{\frac1r-1}$ in origin. 
Herein, we make precise the regularity of $u=u(t)$ on $[0,\9)$ by showing that, if $u_0$ is a probability density  or, more generally, a probability measure in the Wasserstein space $W_p(\rr^2)$ with $1<p<2$, then $u(t)\in \mathbb{W}_p=W_p(\rr^2)$, $\ff t>0$, and the function $t\to u(t)$ is absolutely continuous and H\"older--$\-\(1-\frac1p\)$ continuous 
in the Wasserstein metric $W_p$. This result --  which by our knowledge is new in literature -- was obtained by interpreting the vorticity equation \eqref{e1.7} as a continuity equation of the form $u_t+\divv(v u)=0$, with the $L^p_u$ vector field $v=-\nu\nabla(\log u)+K(u)$. It should mention, however, that for the 2D Euler equation such a result in $W_2$ follows directly by Youdovich's work \cite{14} on the existence and uniqueness for $L^\9$-initial data. As it is well known, the Wasserstein topology is weaker than that given by the total variation distance, but it is, however, more appropriate to represent various PDE as gradient flows and it is closely related to the optimal transport theory (\cite{11b}, \cite{10}). 
The Wasserstein distances were already successfully used in fluid dynamics as a weak measure for vortex concentration to describe the dynamics of vorticity flow with the initial state concentrated around a finite number of distinct points in $\rr^2$ (\cite{5a}). 
 
\n{\bf Notations.}  $L^p(\rr^d),\ 1\le p\le\9$ (denoted $L^p$) is the    space of all Lebesgue mea\-su\-rable and $p$-integrable functions on $\rr^d$, with the standard norm $|\cdot|_p$. 
We shall denote by $\calm(\rr^d)$ the space of all finite Radon measures on $\rr^d$ with the total variation norm denoted by $\|\cdot\|_1$. Given a Banach space $\calx$ and $0\le t_0<T\le\9$, we denote by $C([t_0,T];\calx)$ the space of all   continuous $\calx$-valued functions on $[t_0,T]$. For \mbox{$1\le p\le\9$,} we shall denote by $L^p(t_0,T;\calx)$ the space of $\calx$-valued, $L^p$-Bochner integrable functions on $(t_0,T)$. 
If $X_1,X_2$ are two Banach spaces, we shall denote by  $L(X_1,X_2)$ the space of linear continuous operators from $X_1$ to $X_2$. 
By $\calp(\rr^d)$ we denote the set of all probability measures on $\rr^d$ 
$$\calp(\rrd)=\left\{\mu\in \calm(\rrd);\int_\rrd\mu(dx)=1\right\}$$
 and $\calp^a(\rrd)$ is the set of all probability densities on $\rr^d$, that is, 
\begin{equation}\label{e1.10}
	\calp^a(\rr^d)=\left\{\rho\in L^1(\rr^d);\rho\ge0,
	\mbox{ a.e. in }\rr^d;\int_{\rr^d}\rho(x)dx=1\right\}.
	\end{equation}
The Wasserstein space of order $p$ on $\rrd$, $d\ge1$, denoted by  $\mathbb{W}_p(\rrd)$ is the~set
\begin{equation}\label{e1.11}
\calp_p(\rrd)=\left\{\mu\in\calp(\rrd);\int_\rrd|x|^p\mu(dx)<\9\right\}\end{equation}
equipped with the distance
\begin{equation}\label{e1.12}
\barr{ll}
W_p(\mu,\nu)\!\!\!
&\dd=\(\inf_{\lbb\in\Pi(\mu,\nu)}
\int_\rrd|x-y|^pd\lbb(x,y)\)^{\frac1p}\vsp
&=\dd\inf\left\{\mathbb{E}\(|X-Y|^p\)^{\frac1p};\ \call_X=\mu,\ \call_Y=\nu\right\}.\earr
\end{equation}
Herein, $\Pi(\mu,\nu)$ is the set of all probability measures $\lbb$ on $\rrd\times\rrd$ with first mar\-ginal $\mu$ and second marginal $\nu$, that is, $\lbb(B{\times}\rrd){=}\mu(B)$, \mbox{$\lbb(\rrd{\times} B){=}\nu(B)$} for each Borelian set $B\subset\rrd$; $\mathbb{E}$ stands for expectation and $X,Y$ are random  variables with the laws $\call_X$ and $\call_Y$, respectively. For $p=1$, the Wasserstein distance $W_1$ is derived from the {\it optimal transport problem} (the {\it Monge--Kantorovich problem}) and became very popular in the last decades for its applications in statistical mechanics, stochastic analysis and theory of gra\-dient flows. (See \cite{11b}, \cite{10}.) It should be mentioned that the convergence of a sequence $\{\mu_n\}\subset \mathbb{W}_p(\rrd)$  is the usual weak$^*$-convergence in $\calm(\rrd)$ plus the convergence of $p$-order moments $\left\{\int_\rrd|x|^p\mu_n(dx)\right\}$ which implies the tightness of $\{\mu_n\}$. The distance $W_p(\mu,\nu)$ can be interpreted also as a weak measure of concentration of measure $\mu$ with respect to a given measure $\nu$, a concept with deep implications in fluid dynamics.

Given a function (curve) $u:[t_0,T]\to\mathbb{W}_p(\rrd)$, one defines the {\it speed} (or~{\it metric derivative}) of $u$ at $t$ as (see \cite{9})
\begin{equation}\label{e1.11a}
	|u'|(t)=\lim_{\vp\to0}\frac1{|\vp|}\,{W}_p(u(t+\vp),u(t)),\ t\in[t_0,T],\end{equation}
provided this limit exists. 

The function $u:[t_0,T]\to\mathbb{W}_p(\rrd)$ is said to be {\it absolutely continuous} if there is $g\in L^1(t_0,T)$ such that $W_p(u(t),u(s))\le\int^t_sg(\tau)d\tau$, for all $t_0\le s<t\le T$. We shall denote by $AC([t_0,T];W_p)$ the space of absolute continuous functions $u:[t_0,T]\to\mathbb{W}_p(\rrd)$. The function $u$ is said to be {\it H\"older$-q$ continuous} on $[t_0,T]$ if $W_p(u(t),u(s))\le C|t-s|^q$, for all $t_0<s\le t\le T$. We also note that, if $u\in AC([t_0,T];\mathbb{W}_p(\rrd))$, then $|u'|\in L^1(t_0,T)$ (see \cite{1}, p.~24). Taking into account that $W_p(u(t),u(t+h))$ represents the minimal cost for moving the measure $u(t)dx$ to $u(t+h)dx$, the above concepts of continuity for the function $t\to u(t)$ in the space $\mathbb{W}_p(\rrd)$ should be interpreted in this meaning, that is, of optimal transport processes.

Given a probability density $\rho\in\calp^a(\rrd)$, its {\it Boltzmann--Gibbs entropy}, also called {\it continuous entropy} of $\rho$, is defined as
\begin{equation}\label{e1.13}
S(\rho)=-\int_\rrd\rho(x)\log\rho(x)dx.
\end{equation}
In the following, we shall mainly use definition \eqref{e1.12} in the special case where $d=2$ and $\mu$ and $\nu$ are probability densities on $\rr^2$, that is $\mu=\rho_1dx$, $\nu=\rho_2dx$, where $\rho_1,\rho_2\in\calp^a(\rr^2)$.

\section{Mild solutions to the vorticity equation \eqref{e1.7}}\label{s2}
\setcounter{equation}{0}

We shall briefly recall herein the basic existence results for mild solutions  of equation \eqref{e1.7} 
with the main reference \cite{6} (see, also, \cite{4a}, \cite{5}, \cite{8}).

A function $u\in L^{r_1}_{\rm loc}(0,\9;L^{r_2}),\ r_1,r_2\ge 1$, is called a {\it mild solution to \eqref{e1.7}} if it satisfies the integral equation 
\begin{equation}\label{e2.1}
u(t)=e^{\nu t\Delta}u_0-\divv\int^t_0 e^{\nu(t-s)\D}(K(u(s))u(s))dx,\ t>0,\end{equation} 

\n where $e^{t\D}$ is the heat semigroup in $\rr^2$, which is well defined on all $L^p$,\break $1\le p\le\9$. (If $u_0\in\calm(\rr^2)$, then $\|e^{t\D}u_0\|_{L^p}\le C t^{-1+\frac1p}\|u_0\|_1,$ $\ff t>0$, for all $p>1.$) 

\begin{proposition}\label{p2.1} 
	Assume that $u_0\in\calm(\rr^2)$. Then, equation \eqref{e1.7} has a unique mild solution $u:[0,\9)\to\calm(\rr^2)$ which is $t$-continuous on $[0,\9)$ in the $w^*$-topology of $\calm(\rr^2)$ and satisfies
\begin{itemize}
\item[\rm(i)] $u\in C([0,\9);L^q)$ for $1\le q<\9$,
\begin{equation}\label{e2.2}
|u(t)|_r\le C_rt^{-1+\frac1r}\|u_0\|_1,\ \ff t>0,\ r\in[0,\9).
\end{equation} 	
\end{itemize}
Moreover, the solution $u$ is smooth for $t>0$,
\begin{equation}\label{e2.3}
|D^k_t\,D^j_xu(t,x)|_r\le C_r\,t^{\frac1r-k-\frac j2-1}\|u_0\|_1,\ \ff(t,x)\in(0,\9)\times\rr^2, 
\end{equation} 
for all $r\in[1,\9]$ and $k,j=0,1,...,$ and
\begin{equation}
\label{e2.3a}
\int_{\rr^2}u(t,x)dx=\int_{\rr^2}u_0(dx),\ \ff t>0.
\end{equation}
\end{proposition}

\n The existence part of Proposition \ref{p2.1} was proved by Giga et al. \cite{6} for initial data $u_0$ in $\calm(\rr^2)$ and the uniqueness for $u_0$ with the sufficiently small atomic part $(u_0)_{pp}$. This smallness  condition for uniqueness was removed later by I.~Gallagher and T.~Gallay. In particular, for $u_0\in :L^1$ there is a unique mild solution which satisfies 
\begin{eqnarray}
u\in C(0,\9;L^q),&&\hspace*{-5mm}\ff q\ge1,\label{e2.4}\\[1mm]
|u(t)|_r\le C_r\,t^{\frac1r-1}\,|u_0|_1,&&\hspace*{-5mm}\ff t\ge0,\ r\ge1.\label{e2.5}
\end{eqnarray}	

Moreover, by \eqref{e2.3a} we see that if $u_0\in\calp^a(\rr^2)$, then 
\begin{equation}\label{e2.6}
u(t)\in\calp^a(\rr^2),\ t\ge0,
\end{equation}
and this extends to all $u_0\in\calp(\rr^2)$. 

For $u_0\in\calm(\rr^2)$, the solution $u$ to \eqref{e1.7}  expressed as (see Theorem 4.3 in \cite{6})
\begin{equation}\label{e2.7}
	u(t,x)=\int_{\rr^2}\Gamma(t,x;0,y)(u_0)(dy),\ (t,x)\in(0,\9)\times\rr^2,
\end{equation}
where the function $\Gamma\equiv\Gamma_v(t,x;s,\xi)$ is the fundamental solution to the linear parabolic operator
$$L_v(u)=u_t-\nu\Delta u+(u\cdot\nabla)v,\ (t,x)\in(0,\9)\times\rr^2.$$\newpage
We have
\begin{equation}\label{e2.8}
\barr{c}
\Gamma_v\ge0;\ \dd\int_{\rr^2}\Gamma_v(t,x;s,\xi)d\xi=1,\
\dd\int_{\rr^2}\Gamma_v(t,x;s,\xi)dx=1,\vsp
\hfill 0\le s\le t<\9,\ x\in\rr^2,\vsp 
\dd\lim_{t\downarrow s}\int_{\rr^2} \Gamma_v(t,x;s,\xi)f(\xi)d\xi=f(x),\ \ff f\in C_b(\rr^2).\earr\end{equation}
\begin{equation}\label{e2.9}
	\barr{l}
	C_1(t-s)\1\exp(-C_2|x-y|^2(t-s)\1)
	\le\Gamma(t,x;s,y)\vsp 
	\qquad\le C_3(t-s)\1\exp(-C_4|x-y|^2(t-s)\1),\ t\ge s\ge 0,\earr
\end{equation}
where $C_j,\ j=1,2,3,4$, depends of $\|u_0\|_1$. 

It should be mentioned that the estimates \eqref{e2.2}--\eqref{e2.3}, as well as the conservation of the vorticity formula \eqref{e2.3a}, follow directly by \eqref{e2.7}--\eqref{e2.9}.

Everywhere in the following, by {\it solution to the vorticity equation}  \eqref{e1.7} we mean a mild solution  $u:[0,\9)\times\rr^2\to\rr$, which has the properties mentioned in Proposition \ref{p2.1}, that is,   $t$-continuous in the $w^*$--topology of $\calm(\rr^2)$ and which  satisfies \eqref{e2.2}--\eqref{e2.3a} (\eqref{e2.4}--\eqref{e2.5}, respectively, if $u_0\in L^1$).

\section{The main result}\label{s3}
\setcounter{equation}{0}

Given $u_0\in\calp^a(\rr^2)$, we set
\begin{equation}\label{e3.1}
I(u_0)=\int_{\rr^2}u_0(x)\log(u_0(x))dx=-S(u_0),
\end{equation}
where $S(u_0)$ is the Boltzmann--Gibbs entropy \eqref{e1.13}. We note that the function $I:L^1\to \,]-\9,+\9]$ is convex and $\not\equiv+\9$. Moreover, by Lemma A1 in Appendix, we have for $\alpha\in\(\frac{p+2}2,1\)$,
\begin{equation}\label{e3.2}
I(u_0)\ge-C_\alpha(M_p(u_0)+1)^\alpha,\ \ff u_0\in\calp_p(\rr^2)\cap \calp^a(\rr^2),\end{equation}
where $M_p$ is the moment of order $p$, i.e.,
\begin{equation}\label{e3.3}
M_p(u_0)=\int_{\rr^2}u_0(x)|x|^pdx.\end{equation}

\begin{theorem}\label{t3.1} Let $u_0\in\calp^a(\rr^2)$ be such that
\begin{equation}\label{e3.4}
u_0\in\calp_p(\rr^2),\ |I(u_0)|<\9,\end{equation}
where $1<p<2$. Let $u$ be the  solution to the vorticity equation \eqref{e1.7}. Then, 
\begin{eqnarray}
&u(t)\in \mathbb{W}_p(\rr^2),\ \ff t\ge0,\label{e3.5}\\[1mm]
&u\in AC([0,T];\mathbb{W}_p(\rr^2)),\ \ff T>0.\label{e3.5a}
\end{eqnarray}
Moreover, we have, for $\alpha\in\(\frac{p-2}2,1\)$,
\begin{eqnarray}
\hspace*{-8mm}&\dd\int^T_0\!\!(|u'|(t))^p dt\le C_T\((|I(u_0)|)^{\frac p2}+(M_p(u_0)+1)^{\frac{\alpha p}2}+|u_0|^{p+1}_1\),\ \ff T>0,\quad\label{e3.6}\\
\hspace*{-8mm}&W_p(u(t),u(s))\le(t-s)^{\frac1{p'}}
\(\dd\int^t_s(|u'|(\tau))^pd\tau\)^{\frac1p},\ 0\le s\le t<\9,\label{e3.7}\\[1mm]
\hspace*{-8mm}&\dd\frac d{dt}\,I(u(t))+\nu\int_{\rr^2}\frac{|\nabla u(t,x)|^2}{u(t,x)}\,dx\le0,\ \mbox{a.e. }t>0,\label{e3.8}\\
\hspace*{-8mm}&\dd\frac d{dt}\,S(u(t))\ge4\pi\exp(-S(u(t))),\mbox{ a.e }t>0.\label{e3.9} 
\end{eqnarray}
\end{theorem}
\n Here, $|u'|(t)$ is the speed (metric derivative) of the curve $u:[0,\9)\to\mathbb{W}_p(\rr^2)$ (see \eqref{e1.11a}),  and $\frac1{p'}=1-\frac1p$.

 Theorem \ref{t3.1} extends to  solutions to \eqref{e1.7} with probabilistic measures initial data $u_0$. Namely, we have 
 \begin{theorem}\label{t3.2} Let $u_0\in\calp_p(\rr^2)\cap\calm(\rr^2)$, $1<p<2$,  
 	and let $u$ be a  solution to the vorticity equation \eqref{e1.7}. Then, for all $0<\delta<T<\9$,
 	
 \begin{equation}\label{e3.11}
 	u\in AC([\delta,T];\mathbb{W}_p(\rr^2))\end{equation}
 and $u=u(t)$ is H\"older$-\frac1{p'}$ continuous on $[\delta,T]$ in the Wasserstein metric $W_p$. Moreover, 
 \eqref{e3.7}--\eqref{e3.9} hold on every interval $[\delta,T]$.
 \end{theorem}

 Theorem \ref{t3.2} follows by Theorem \ref{t3.1} by the smoothing effect of the vor\-ti\-city flow $t\to u(t)$ mentioned in Proposition \ref{p2.1}, part \eqref{e2.3}, but the complete proof will be given in Section \ref{s4}. A typical example herein is $u_0=\delta_{x_0}$, 
 which generates a unique solution $u=\frac1{4\pi t}\exp\(-\frac{|x-x_0|^2}{4t}\)$ to \eqref{e1.7}, the so called {\it Lamb--Oseen vortex} (see, e.g., \cite{5}).
 
 One might ask if Proposition \ref{p2.1} remains correct  for $u_0\in\calm(\rr^2)\cap \calp_p(\rr^2)$. However, a simple calculation shows that, if $u_0$ is an atomic measure of the form  $u_0=\sum\limits^N_{j=1}a_j\delta_{x^0_j}$ where $a_j\ge0$, $\sum\limits^N_{j=1}a_j=1$ and $N>1$, this is not true. Indeed, if $u$ is the corresponding solution to \eqref{e1.7}, that is the vorticity starting from the joint-vortex measure, we have by \eqref{e2.7}--\eqref{e2.8}
$$\barr{ll}
W_p^p(u(t),u_0)\!\!\!
 &=\dd\sum^N_{j=1}a_j\int_{V_j}u(t,x)|x-x^0_j|^pdx\vsp 
 &=\dd\sum^N_{j,k=1}a_ja_k\int_{V_j}|x-x^0_j|^p\Gamma(t,x;0,x^0_k)dx\vsp
 &=t\dd\sum^N_{j,k=1}a_ja_k\int_{V_j}|x^0_k-x^0_j+\sqrt{t}\,z|^p\Gamma(t,x_k+\sqrt{t}\,z;0,x^0_k)dz.\earr$$
 Here, $V_j$ are Voronoi power cells of the form
 $$V_j=\left\{x\in\rr^2;\frac1p|x-x_j|^p-\psi_j\le\frac1p|x-x_i|^p-\psi_i,\ff i\ne j\right\},\ j=1,...,N,$$
 where $\psi_j\ge0,\ j=1,...,N$. (See \cite{11b}, p.221.)
 
 Taking into account that by \eqref{e2.9}, we have
 $$\Gamma(t,x^0_k+\sqrt{t}\,z;0,x^0_k)=0\(\frac1t\exp(-|z|^2)\),$$
 it follows that if $N>1$, then the function $t\to W_p(u(t),u_0)$ is not absolutely continuous in a neighbourhood of origin. (However, it is if $N=1$.)
  
Theorem \ref{t3.1}, as well as Theorem \ref{t3.2}, 
identify the vorticity flow \mbox{$u=u(t)$} to the Navier--Stokes equation \eqref{e1.1} as an absolutely continuous and H\"older--continuous curve in the Wasserstein space $\mathbb{W}_p(\rr^2)$, where $p\in(1,2)$. In~particular, it follows that the vorticity flow $t\to u(t)$ is smooth on $[0,\9)$  in the  $\mathbb{W}_p$--topology, which is stronger than the $w^*$--topology.

Apparently, on $(0,\9)$ this property is implied by the $C^\9$-regularity \eqref{e2.1} of the function $t\to u(t)$, but this is not true because the Wasserstein metric $\mathbb{W}_p$ is not dominated on large distances of $R^2$ by the $L^r$--metric, $r\ge1$.

As a matter of fact, as already emphasized, the Wasserstein distances are closely related to optimal transport, processes and it is also suitable to describing instantaneous interactions at large spatial distance in $\rrd$.  
In particular, it is suitable to represent the dynamics which have an infinite speed of propagation, as is the case with the vorticity equation. A nice illustration in this direction is the work \cite{5a} on the evolution and persistence of concentrated vortices in 2D Navier--Stokes flows. In this case, the Wasserstein distance is used to modelling the concentration of the vorticity field that evolves from an initial finite number of point mass vortices. From this perspective, the weak regularity given by Theorem \ref{t3.1} describes the smooth evolution of the vorticity flow $u(t)$ measured on large spatial distances. From the physical point of view, this behaviour suggests that the motion of fluid density does not exhibit infinite speed shocks. 

Inequality \eqref{e3.9} obtained in the proof of Theorem \ref{t3.1} as a byproduct of analysis on the continuity equation equivalent with \eqref{e1.7}, gives the time variation of Boltzmann--Gibbs entropy (entropy production) for the vorticity flow corresponding to the 2D Navier--Stokes equation.

\begin{remark}\label{r3.3} \rm In Theorem \ref{t3.1} and Theorem \ref{t3.2} the case $p=2$, which corresponds to the $2-$Wasserstein distance $W_2$ remains open. 	\end{remark}
	
\begin{remark}\label{r3.4}\rm Arguing as in \cite[Theorem 2.15]{1a},  it follows that the vorticity flow $u=u(t)$ given by Theorem \ref{t3.1} is alternatively generated by a continuity equation of the form \eqref{e4.21} with a Borel vector field $v^*=v^*(t)\in L^2(u(t)dx;\rr^2)$ such that
	$$\barr{c}
	\(\dd\int_{\rr^2}|v^*(t,x)|^pu(t,x)dx\)^{\frac1p}=|u'|(t), \mbox{ a.e. }t>0,\vsp 
	v^*(t)\in\ov{\{\nabla\vf;\ \vf\in C^\9_0(\rrd)\}}^{\,L^p(u(t)dx,\rr^2)}.\earr$$
	This vector field can be defined in variational terms as that which minimizes the instantaneous kinetic energy of the system and can be equivalently expressed in terms of  the tangent space to $\calp_p(\rr^2)$ at $u(t)$.	It should be emphasized, however, that $v^*$ is different from the vector field 
	$$v=-\nu\,\frac{\nabla u}u+K(u),$$
	arising in the representation of the vorticity equation \eqref{e1.7} as a continuity equation. (See \eqref{e4.21}.) From the physical point of view, this means that the motion of the curve $\{u(t);t\ge0\}$ avoids gradient-like movements, filtering so the irrotational vortices. 
\end{remark}

\section{Proof of Theorem \ref{t3.1}}\label{s4}
\setcounter{equation}{0}

We shall prove Theorem \ref{t3.1} following several steps The first one is

\begin{claim}\label{claim1} $u(t)\in\calp_p(\rr^2),\ \ff t\ge0.$
\end{claim}
\begin{proof} By \eqref{e2.7}--\eqref{e2.9} we have	 
	
\begin{equation}\label{e4.1a}
\barr{l}
\dd\int_{\rr^2}u(t,x)|x|^pdx
=\dd\int_{\rr^2\times\rr^2}
\Gamma(t,x,0,y)|x|^pu_0(y)dydx\\
 \qquad
\le Ct\1\dd\int_{\rr^2}|x|^p\exp(-C|x-y|^2t\1)u_0(y)dydx\vsp \qquad
\le Ct\1\dd\int_{\rr^2}u_0(y)dy\int_{\rr^2}(|y|^p+|x-y|^p)\exp(-C|x-y|^2t\1)dx\vsp\qquad
\le Ct\1\dd\int_{\rr^2}|y|^pu_0(y)dy\int_{\rr^2}\exp(-D|x-y|^2t\1)dx\vsp\qquad
+Ct\1\dd\int_{\rr^2}u_0(y)dy\int_{\rr^2}|x-y|^p\exp(-C|x-y|^2t\1)dx.\earr
\end{equation} 
(We have denoted by the same symbol $C$ several positive constants independent of $u_0$.) 
Taking into account that $u_0\in\calp_p(\rr^2)\cap\calp^a(\rr^2)$, we get
\begin{equation}\label{e4.1}
\barr{l}
\dd\int_{\rr^2}u(t,x)|x|^pdx
\le C\dd\int_{\rr^2}u_0(y)|y|^pdy
\int_{\rr^2}|z|^2\exp(-C|z|^2)dx\\\qquad 
+Ct^{\frac p2}\dd\int_{\rr^2}|z|^2\exp(-C|z|^2)dz
\le C(M_p(u_0)+t^{\frac p2}),\ \ff t\ge0,\earr\end{equation} 
as claimed.	
\end{proof}

\begin{claim}\label{claim2} We have
\begin{equation}\label{e4.3}
I(u(t))+\nu\int^t_0\int_{\rr^2}
\frac{|\nabla u(s,x)|^2}{u(s,x)}\,dxds\le I(u_0),\ \ff t\ge0.\end{equation} 	
\end{claim}

\begin{proof} Let $\vf_n(x)=\eta\(\frac{|x|^2}n\),\ x\in\rr^2,$ where $\eta$ is the cut-off function
\begin{equation}\label{e4.4}
\eta\in C^2([0,\9)),\ \eta(r)=1\ \ff r\in[0,1],\ \eta(r)=0\ \ff r\ge2,\ 0\le\eta(r)\le1. \end{equation}
Taking into account that by \eqref{e2.3} $u$ is smooth on $(0,\9)\times\rr^2$, by \eqref{e1.7} we have, for all $\vp>0$,

$$\barr{r}
\dd\int_{\rr^2}u_t(t,x)\log(u(t,x)+\vp)\vf_n(x)dx
+\dd\int_{\rr^2}\nabla u(t,x){\cdot}\nabla(\vf_n(x)\log(u(t,x)+\vp))dx\\
=\dd\int_{\rr^2}(K(u(t,x)){\cdot}\nabla u(t,x))\vf_n(x)\log (u(t,x)+\vp)dx,\ \ff t>0,
\earr$$ 
because $\divv\,K(u(t,x))=0\mbox{ on }(0,\9)\times\rr^2$.
This yields

\begin{equation}\label{e4.5}
	\barr{l}
\dd\frac d{dt}\int_{\rr^2}h(u(t,x)+\vp)\vf_n(x)dx
	+\dd\int_{\rr^2}\frac{|\nabla u(t,x)|^2}{u(t,x)+\vp}\,\vf_n(x)dx\\
	\qquad\qquad+\dd\int_{\rr^2}\nabla h(u(t,x)+\vp){\cdot}\nabla\vf_n(x)dx\vsp
	\qquad\qquad=\dd\int_{\rr^2}(K(u(t,x)){\cdot}\nabla h(u(t,x)+\vp))\vf_n(x)dx,\\
\qquad\qquad=-\dd\int_{\rr^2}h(u(t,x)+\vp))
K(u(t,x)){\cdot}\nabla\vf_n(x)dx,\mbox{ a.e. } \ff t>0,
\earr	
\end{equation}
where $h$ is the function
$$h(r)=r(\log r-1),\ \ff r>0.$$
We have
\begin{equation}\label{e4.6}
\barr{rcll}
|\nabla\vf_n(x)|&\le&\dd\frac{2|\eta'|_\9}{\sqrt{n}},\ 
\nabla\vf_n(x)=0,&\ff x\in\Sigma_n,\vsp 
|\Delta\vf_n(x)|&\le&\dd\frac4n(|\eta"|_\9+|\eta'|_\9),
&\ff x\in\Sigma_n,\earr
\end{equation} 
where 
$\Sigma_n=\{x;\ 0\le |x|\le\sqrt{n}\}\cup\{x;\ |x|>2\sqrt{n}\}.$ 

We set
\begin{eqnarray}
H^\vp_1(t)
&=&-\dd\int_{\rr^2}\nabla h(u(t,x)+\vp){\cdot}\nabla\vf_n(x)dx\label{e4.7}\vsp 
&=&\dd\int_{\rr^2}h(u(t,x)+\vp)\Delta\vf_n(x)dx,\ t>0,\nonumber\\[1mm]
H^\vp_2(t)&=&-\dd\int_{\rr^2}h(u(t,x)+\vp)K(u(t,x)){\cdot}\nabla\vf_n(x)dx,\ t>0.\label{4.7a}
\end{eqnarray} 
By \eqref{e4.5} we have
\begin{equation}\label{e4.8}
\hspace*{-5mm}\barr{r}
\dd\frac d{dt}\!\int_{\rr^2} \!\!h(u(t,x)+\vp)\vf_n(x)dx
+\!\!\int_{\rr^2}\frac{|\nabla u(t,x)|^2}{u(t,x)+\vp}\vf_n(x)dx
=H^\vp_1(t)+H^\vp_2(t),\\ \ff t>0.\earr\end{equation} 
Taking into account that, for some $0<\alpha<1$ 
and $C>0$,
$$|h(r)|\le C\(r^\alpha I_{[0<r\le1]}(r)+r^{\alpha+1}I_{[r\ge1]}(r)\), \ \ff r\ge0,$$
and that by \eqref{e2.2}
$$|u(t)|_{1+\alpha}\le Ct^{-\frac\alpha{\alpha+1}}|u_0|_1,\ \ff t>0,$$
it follows by the Lebesgue dominated  convergence theorem that
\begin{equation}\label{e4.9}
\lim_{\vp\to0}H^\vp_1(t)=-\int_{\rr^2}u(t,x)\log u(t,x)\Delta\vf_n(x)dx,\ \ff t>0.\end{equation} 
Similarly, taking into account that by \eqref{e1.6}
$$|K(u(t))|_q\le C|u(t)|_r\le Ct^{\frac 1r-1}|u_0|_1,$$
for all $r\in(1,2)$ and $\frac1q=\frac1r-\frac12$, it follows that
\begin{equation}\label{e4.10}
\lim_{\vp\to0}H^\vp_2(t)=-\int_{\rr^2}u(t,x)
\log(u(t,x))(K(u(t,x){\cdot}\nabla \vf_n(x)))dx,\ \ff t>0.
\end{equation} 
By \eqref{e4.8}, we have
\begin{equation}\label{e4.10a}
\barr{l}
\dd\int_{\rr^2}h(u(t,x)+\vp)\vf_n(x)dx+\nu
\dd\int^t_0\dd\int_{\rr^2}
\frac{|\nabla u(s,x)|^2}{u(s,x)+\vp}\vf_n(x)dxds\vsp 
\qquad+\dd\int_{\rr^2}h(u_0(x)+\vp)\vf_n(x)dx
+\int^t_0(H^\vp_1(s)+H^\vp_2(s))ds,\ \ff t\ge0.
\earr\end{equation}
We note that, for $0<\alpha<1$, we have
$$z\log z\ge-C_\alpha z^\alpha,\ \ff z\ge0,$$
and this yields
$$h(u(t,x)+\vp)+C_\alpha(u(t,x)+\vp)^\alpha\ge0,\ \mbox{a.e. }(t,x)\in(0,\9)\times\rr^2.$$
Then, by Fatou's lemma we have
\begin{equation}\label{e4.10aa}
\hspace*{-4mm}\barr{l}
\dd\liminf_{\vp\to0}
\(\dd\int_{\rr^2}h(u(t,x)+\vp)\vf_n(x)dx
+C_\alpha\dd\int_{\rr^2}(u(t,x)+\vp)^\alpha\vf_n(x)dx\)
\vsp 
\ge\dd\int_{\rr^2}h(u(t,x))\vf_n(x)dx
+C_\alpha\int_{\rr^2}(u(t,x))^\alpha\vf_n(x)dx,\ \ff t>0.\earr\end{equation}
On the other hand, since
$$0\le(u(t,x)+\vp)^\alpha\vf_n(x)\le((u(t,x))^\alpha+1)\vf_n(x),\ \ff(t,x)\in(0,\9)\times\rr^2,$$
and $$(u(t))^\alpha\vf_n\in L^1,\ \ff t>0,$$
we have
$$\lim_{\vp\to0}\int_{\rr^2}(u(t,x)+\vp)^\alpha\vf_n(x)dx
=\int_{\rr^2}(u(t,x))^\alpha\vf_n(s)dx,\ \ff t>0,$$
and so, \eqref{e4.10aa} yields
$$\liminf_{\vp\to0}\int_{\rr^2}h(u(t,x)+\vp)\vf_n(x)dx\ge\int_{\rr^2}h(t,x))\vf_n(x)dx.$$
Then, letting $\vp\to0$ in \eqref{e4.10a}, by \eqref{e4.9}--\eqref{e4.10} we get that
\begin{equation}\label{e4.11}
\hspace*{-4mm}\barr{l}
\dd\int_{\rr^2}u(t,x)\log(u(t,x))\vf_n(x)dx
+\nu\dd\int^t_0\int_{\rr^2}
\frac{|\nabla u(s,x)|^2}{u(s,x)}\vf_n(x)dxds\vsp
\qquad\le\dd\int_{\rr^2}h(u_0(x))\vf_n(x)dx
-\int^t_0\dd\int_{\rr^2}u(s,x)\log(u(s,x))(\Delta\vf_n(x)\mk\vsp 
\qquad+K(u(s,x)){\cdot}\nabla\vf_n(x))dxds,\,\ff t\ge0.
\earr
\end{equation}

Now, if we let $n\to\9$ in \eqref{e4.11}, we get by \eqref{e4.6}

\begin{equation}\label{e4.12}
	\barr{r}
\dd\int_{\rr^2}u(t,x)\log(u(t,x))dx+\nu\int^t_0
\int_{\rr^2}\frac{|\nabla u(s,x)|^2}{u(s,x)}dxds\vsp
 \le\dd\int_{\rr^2}u_0\log(u_0)dx=I(u_0)<\9,\ \ff t\ge0,\earr\end{equation} 
because, by \eqref{e3.2}, we have

\begin{equation}\label{e4.13}
\hspace*{-5mm}\barr{l}
\dd\int^t_0\dd\int_{\rr^2}
|u(s,x)\log(u(s,x))\Delta\vf_n(x)|dx\vsp 
\quad\le\dd\int^t_0\dd\int_{\rr^2}((u(s,x)\log u(t,x))^+
+(u(t,x)\log(u(t,x)))^-)|\Delta\vf_n(x)|dx\vsp 
\quad\le C\dd\int^t_0\int_{\rr^2}
((u(s,x))^\alpha+(u(s,x))^{\alpha+1})|\Delta\vf_n(x)|dx,
\earr\end{equation}
for $\alpha\in(0,1)$. By \eqref{e4.6}, we have
$$\barr{ll}
\dd\int_{\rr^2}(u(s,x))^\alpha|\Delta\vf_n(x)|dx\!\!\!
&\le\dd\(\int_{\Sigma_n}u(s,x)dx\)^\alpha
\(\int_{\Sigma_n}|\Delta\vf_n(x)|^{1-\alpha}dx\)^{\frac1{1-\alpha}}
\vsp
&\le C\(\dd\int_{\Sigma^c_n}u(s,x)dx\)^1,\earr$$
where $\Sigma^c_n=\{x;\sqrt{n}\le|x|\le2\sqrt{n}\}$. Since $u\in L^1((0,T)\times\rr^2)$, $\ff T>0$, we infer that
\begin{equation}\label{e4.14a}
\lim_{n\to\9}\int^t_0\int_{\rr^2}(u(s,x))^\alpha|\Delta\vf_n(x)|dxds=0.\end{equation}
We also have by \eqref{e4.7} and \eqref{e2.2},
\begin{equation}\label{e4.14aa}
 \int^t_0\!\!\!\int_{\rr^2}(u(s,x))^{\alpha+1}
|\Delta\vf_n(x)|dxds
\le\dd\frac Cn\int^t_0\!\!\!|u(s)|^{\alpha+1}_{\alpha+1}ds
\le\dd\frac Cn\,t^{1-\alpha}|u_0|^{\alpha+1}_1.
\end{equation}
Then, by \eqref{e4.13}--\eqref{e4.14aa} we see that
\begin{equation}\label{e4.14aaa}
\dd\lim_{n\to\9}	\int^t_0\int_{\rr^2}|u(s,x)\log u(s,x)|\,|\Delta\vf_n(x)|dx=0. 
\end{equation}
Similarly, by \eqref{e1.6} and \eqref{e2.2} we have 

\begin{equation}\label{e4.14}
\hspace*{-4mm}\barr{l}
\dd\int^t_0\int_{\rr^2}|u(s,x)\log(u(s,x))|\,
|K(u(s,x)){\cdot}\nabla\vf_n(x)|dxds\vsp
\qquad\le\dd\frac Cn\int^t_0\dd\int_{\rr^2}
(|u(s,x)|^\alpha+|u(s,x)|^{\alpha+1})
|K(u(s,x))|dx\vsp
\qquad\le\dd\frac Cn\int^t_0(|K(u(s))|_{q_1}
|u(s)|^\alpha_{\alpha q'_1}
+|K(u(s,x))|_{q_2}|u(s)|^{\alpha+1}_{(\alpha+1)q'_2})ds
\vsp
\qquad\le\dd\frac Cn\int^t_0
\(s^{\frac1{m_1}-\frac1{q_1}-\alpha}
+s^{\frac1{m_2}-\frac1{q_2}-\alpha-1}\)ds
\le \dd\frac Cn\,t^a,\earr
\end{equation} 
where $\frac1{q_i}=\frac1{m_i}-\frac12$, $m_i\in(1,2),\ \frac1{q'_i}=1-\frac1{q_i}$, $i=1,2$, and $\alpha\in(0,1)$. Then, for $\alpha$ and $m_i$ suitably chosen, $a>0$, and so for $n\to\9$ the left side parts of \eqref{e4.13} and \eqref{e4.14} go to zero and, therefore,  \eqref{e4.12} (equivalently \eqref{e4.3}) follows as claimed.

Taking into account that by \eqref{e3.2} and \eqref{e4.1}
\begin{equation}\label{e4.15}
I(u(t))\ge-C_\alpha(M_p(u(t)+1))^\alpha\ge-C^1_\alpha(M_p(u_0)+1+t^{\frac p2})^\alpha,\ \ff t\ge0,
\end{equation}
 for $\alpha\in\(\frac{p+2}2,1\)$, it follows by \eqref{e4.12}
\begin{equation}\label{e4.16}
\nu\int^t_0\int_{\rr^2}
\frac{|\nabla u(s,x)|^2}{u(s,x)}dx
\le t(|I(u_0)|+C^1_\alpha(M_p(u_0)+1+t^{\frac p2})^\alpha),\ \ff t\ge0,\end{equation} 
for $\alpha\in\(\frac{p+2}2,1\)$.

Now, coming back to \eqref{e4.8} we get by \eqref{e4.9}--\eqref{e4.10}  that
$$\barr{l}
\dd\lim_{\vp\to0}
\(\frac d{dt}\int_{\rr^2}h(u(t,x)+\vp)\vf_n(x)dx
+\nu\dd\int_{\rr^2}
\frac{|\nabla u(t,x)|^2}{u(t,x)+\vp}\vf_n(x)dx\)\vsp 
\qquad=-\dd\int_{\rr^2}u(t,x)\log (u(t,x))(\Delta\vf_n(x)+K(u(t,x)){\cdot}\nabla\vf_n(x))dx,\ \ff t>0.
\earr$$
Since, by \eqref{e4.16}, by virtue of the Lebesgue dominated convergence theorem,
$$\dd\lim_{\vp\to0}
\dd\int_{\rr^2}\frac{|\nabla u(t,x)|^2}{u(t,x)+\vp}\vf_n(x)dx
=\dd\int_{\rr^2}
\frac{|\nabla u(t,x)|^2}{u(t,x)}\vf_n(x)dx,\mbox{ a.e. }t>0,$$
we get
$$\barr{ll}
\dd\lim_{\vp\to0}
\(\frac d{dt}\int_{\rr^2}h(u(t,x)+\vp)\vf_n(x)dx\)
+\nu\dd\int_{\rr^2}
\frac{|\nabla u(t,x)|^2}{u(t,x)}\vf_n(x)dx\vsp
\quad=-\dd\int_{\rr^2}(u(t,x)\log(u(t,x))
(\Delta\vf_n(x)+K(u(t,x)){\cdot}\nabla\vf_n(x)))dx,\mbox{ a.e. }t>0,
\earr$$
which yields
$$\barr{l}
\dd\frac d{dt}
\int_{\rr^2}u(t,x)\log(u(t,x))\vf_n(x)dx
+\nu\dd\int_{\rr^2}
\frac{|\nabla u(t,x)|^2}{u(t,x)}\vf_n(x)dx\vsp
\qquad=-\dd\int_{\rr^2}u(t,x)\log(u(t,x))(\Delta\vf_n(x)+K(u(t,x)){\cdot}\nabla\vf_n(x))dx,\mbox{ a.e.}t>0.
\earr$$
Hence, by \eqref{e4.12} it follows that
$$\lim_{u\to\9}\frac d{dt}\int_{\rr^2}u(t,x)\log u(t,x)
\vf_n(x)dx=\frac  d{dt}\,I(u(t))\mbox{ in }\cald'(0,\9).$$
Then, for $n\to\9$ we infer that $\frac{dI}{dt}\in L^1(0,\9)$ and, therefore, the function $t\to I(u(t))$ is absolutely continuous on $[0,\9)$ and \eqref{e3.8} holds, that is,
$$\frac d{dt}\,I(u(t))+\nu\int_{\rr^2}\frac{|\nabla u(t,x)|^2}{u(t,x)}\,dx=0,\ \mbox{ a.e. }t>0.\vspace*{-10mm}$$
\end{proof}	\sk

\begin{claim}\label{claim3} We have, for all $t\ge0$,
\begin{equation}\label{e4.17}
\nu\!\!\int^t_0\!\!\int_{\rr^2}
\frac{|\nabla u(s,x)|^p}{(u(s,x))^{p-1}}dxds
\le t^{\frac p2+\frac2{2-p}}\,\nu^{\frac{2-p}2}
\(|I(u_0)|{+}C^1_\alpha(M_p(u_0)
{+}^{\frac p2})^\alpha\)^{\frac p2}\!\!.	
	\end{equation} 	
\end{claim}

\begin{proof} 
By H\"older's inequality, we have
$$\barr{rcl}
\nu\dd\int^t_0\(\int_{\rr^2}
\frac{|\nabla u(s,x)|^p}{(u(s,x))^{p-1}}dx\)ds
&\le&
\nu\dd\int^t_0\(\int_{\rr^2}
\frac{|\nabla u(s,x)|^2}{u(s,x)}dx\)^{\frac p2}\vsp 
\(\dd\int_{\rr^2}u(s,x)dx\)^{\frac{2-p}2}ds&=&
\nu\,t^{\frac2{2-p}}
\(\dd\int^t_0\int_{\rr^2}
\frac{|\nabla u(s,x)|^2}{u(s,x)}dxds\)^{\frac p2},
\earr$$	
and so, by \eqref{e4.16} we get \eqref{e4.17}. 
\end{proof} 

\begin{claim}\label{claim4} We have, for all $p\in[1,2),$
	\begin{equation}\label{e4.18}
		\int^t_0\int_{\rr^2}|K(u(s,x))|^pu(s,x)dxds
		\le C_p\,t^{1-\frac p2}|u_0|^{p+1}_1,\ \ff t\ge0.
	\end{equation} 	
\end{claim}

\begin{proof} We have, for all $r>1$,
\begin{equation}\label{e4.19}
	\int_{\rr^2}|K(u(s,x))|^pu(s,x)dx
	\le|K(u(s,x))|^p_{pr'}|u(s)|_r,\ \ff s>0,
\end{equation} 
where $\frac1{r'}=1-\frac1r$. By \eqref{e1.6} we have, for all $p\in[1,2),$
\begin{equation}\label{e4.20}
|K(u(s))|^p_{pr'}\le C|u(s)|^p_{m},\ \ff s\ge0,
\end{equation}  
where $\frac1m=\frac1{pr'}+\frac12=\frac1p\(1-\frac1r\)+\frac12,\ m\in(1,2).$\mk

Then, by \eqref{e2.2} and \eqref{e4.19}--\eqref{e4.20} we get
$$\int_{\rr^2}|K(u(s,x))|^pu(s,x)dx\le C\,s^{-\frac p2}|u_0|^{p+1}_1,\ \ff s>0,$$
and so \eqref{e4.18} follows.
\end{proof}

\begin{proof}[Proof of Theorem {\rm\ref{t3.1}}] (continued) We shall write the vorticity equation \eqref{e1.7} as the continuity equation  
\begin{equation}\label{e4.21}
	\barr{ll}
	u_t+\divv(vu)=0&\mbox{in }(0,\9)\times\rr^2,\vsp
	u(0,x)=u_0(x),&x\in\rr^2,\earr
\end{equation} 
where $v=v(t,x)$ is the velocity vector field
\begin{equation}\label{e4.27a}
	v(t,x)=-\nu\,\frac{\nabla u(t,x)}{u(t,x)}+K(u(t,x)),\ (t,x)\in(0,\9)\times\rr^2.\end{equation}  
By \eqref{e4.18}--\eqref{e4.19} we know that
\begin{equation}\label{e4.22}
\barr{r}
\dd	\int^T_0\!\!\!\int_{\rr^2}\!|v(t,x)|^pu(t,x)dx\le C_T(|I(u_0)|^{\frac p2}{+}(M_p(u_0){+}1)^{\frac{\alpha p}2}{+}|u_0|^{p+1}_1),\\ \ff T>0.\earr
\end{equation} 
By Theorem 10.2.2 in \cite{1} (see, also, Theorem 2.15 in \cite{1a} and Theorem 4.6 in \cite{9}) it follows by \eqref{e4.22} that, for all $T>0$, the function $u:[0,T]\to\mathbb{W}_p(\rr^2)$ is absolutely continuous and 
$$|u'|(t)\le\(\int_{\rrd}|v^*(t,x)|^pu(t,x)dx\)^{\frac1p},\mbox{ a.e. }t\in(0,T).$$
Then, by \eqref{e4.22},  it follows \eqref{e3.6}, as claimed. As regards \eqref{e3.7}, which, in particular,  implies the H\"older continuity in $\mathbb{W}_p(\rr^2)$ of $u=u(t)$, it follows by \eqref{e3.7} and the inequality 
$$W_p(u(t),u(s))\le\int^t_s|u'|(\tau)d\tau,\ \ff0<s\le t<T.$$ 
 Next, \eqref{e3.9} follows by \eqref{e3.8} and the classic logarithmic Sobolev inequality  (see,~e.g., \cite{4})
$$\exp\(\int_{\rr^2}z(x)\log z(x)dx\)\le(4\pi e)\1\int_{\rr^2}\frac{|\nabla z(x)|^2}{z(x)}\,dx,$$
for all $z\in H^1(\rr^2)\cap\calp^a(\rr^2)$.\sk 
\end{proof}

\section{Proof of Theorem \ref{t3.2}} 
	\setcounter{equation}{0}

 As mentioned earlier, Theorem \ref{t3.2} is a direct consequence of smoothing effect of the vorticity flow $t\to u(t)$ on the initial data $u_0\in\calm(\rr^2)$. Namely, let $u_0\in\calp_p(\rr^2)$ and let $u$ be the corresponding  solution to \eqref{e1.7}. Let $\delta>0$ be ar\-bi\-tra\-rily small, but fixed. As seen in Proposition~\ref{p2.1}, $u=u(t,x)$ is smooth on $(0,\9)\times\rr^2$ and, in particular, it follows that $u(\delta)\in\calp^a(\rr^2)$. Moreover, by the integral representation formula \eqref{e2.7} it follows as in the proof of Claim \ref{claim1} (see \eqref{e4.1a}--\eqref{e4.1}) that $u(t)\in\calp_p(\rr^2)$ for all $t\ge0$ and, therefore, $u(\delta)\in\calp_p(\rr^2)$. Furthermore, by \eqref{e2.2} it follows that $u(\delta)\in L^{\alpha+1}$ for all $\alpha\in[0,\9)$ and this implies (see Remark \ref{r5.1} in Appendix) that $|I(u(\delta))|<\9$. Then, applying Theorem \ref{t3.1} with the initial data $u(\delta)$ on the interval $(\delta,\9)$, we get the desired result.\hfill$\Box$

\section{Appendix}\label{s5}
\setcounter{equation}{0}

\n{\bf Lemma A.1.} {\it Let $I:L^1\to\,]-\9,+\9]$ be the function \eqref{e3.1}. Then, for each $\alpha\in\(\frac{p+2}2,1\)$ there is 
	$C_\alpha>0$ such that
	\begin{equation}\label{e5.1}
		I(u_0)\ge- C_\alpha(M_p(u_0)+1)^\alpha,\ \ff u_0\in D(I)=\calp^a(\rr^2)\cap \calp_p(\rr^2).\vspace*{-6mm}
\end{equation}}
\begin{proof}
	The argument is similar to that used in \cite{7} for $p=2$, and so it will be sketched only. Namely,  we note first that, for $\alpha\in(0,1)$,
	$$(z\log z)^-\le C^1_\alpha\,z^\alpha,\ \ff z\ge0,$$	where $C^1_\alpha$ is independent of $z$. This yields, for $u_0\in D(I),$
\begin{equation}\label{e5.2}
\hspace*{-4mm}	\barr{l}
	\dd\int_{\rr^2}(u_0(x)\log u_0(x))^-dx
	\le C^1_\alpha\dd\int_{\rr^2}u^\alpha_0(x)dx\vsp 
	\qquad\le C^1_\alpha\(\dd\int_{\rr^2}u_0(x)(1+|x|^p)dx\)^{\alpha}
	\(\dd\int_{\rr^2}(1+|x|^p)^{-\frac\alpha{1-\alpha}}dx\)^{1-\alpha}\vsp
	\qquad\le C^1_\alpha(M_p(u_0)+1)^\alpha
	\(\dd\int^\9_0\!\!
	\frac{r\,dr}{(1+r^p)^{\frac\alpha{1-\alpha}}}\)^{1-\alpha}\!\!
	\le C^1_\alpha\,C^2_\alpha(M_p(u_0)+1)^\alpha,\earr\hspace*{-4mm}
\end{equation}
	where $0<C^2_\alpha<\9$ for $\alpha\in\(\frac2{2+p},1\)$. 
	
	Then, taking into account that $$I(u_0)=\int_{\rr^2}((u_0\log u_0)^+-(u\log u_0)^-)dx,$$ \eqref{e5.1} holds.
\end{proof}
 
\begin{remark}\label{r5.1} \rm In particular, it follows by \eqref{e5.2} that
	$$|I(u_0)|=\int_{\rr^2}((u_0\log u_0)^++(u_0\log u_0)^-)dx<\9$$
	if $u_0\in\calp_p(\rr^2)\cap L^{a+1}$ for some $a>0.$
	\end{remark}


\end{document}